\documentclass{amsart}

\usepackage{amsthm} 
\usepackage{color}

\usepackage{graphicx,epsfig}
\usepackage{epstopdf}
\usepackage{amsmath, amssymb, latexsym, euscript}
\usepackage{url}
\usepackage[all]{xy}
\usepackage{psfrag}

\def\smallskip{\vspace{.15cm}}
\def\medskip{\vspace{.3cm}}
\def\text{\mbox}
\def\rh2{{\mathbb R}{\mathbb H}^2}
\def\ch2{{\mathbb C}{\mathbb H}^2}

\def\RP2{{\mathbb{RP}}^2}
\def\RP3{{\mathbb{RP}}^3}

\def\PGL3{PGL(3{,\mathbb R})}
\def\PGL4{PGL(4{,\mathbb R})}

\def\H2R{{\mathbb H}^2\times {\mathbb R}}

\def\interior{{\rm int}\thinspace}

\def\cl{\rm{cl}\thinspace}

\newtheorem{theorem}{Theorem}[section]
\newtheorem{lemma}[theorem]{Lemma}

\newtheorem{proposition}[theorem]{Proposition}

\newtheorem{addendum}[theorem]{Addendum}

\begin{document}

\title{Conservative Subgroup Separability for  Surfaces with Boundary.}
\author{Mark D. Baker}
\author{Daryl Cooper}
\today

\address{IRMAR,
Universit\'e de Rennes 1,
35042 Rennes Cedex, FRANCE}
\address{Department of Mathematics, University of California, Santa Barbara, CA 93106, USA}

\email[]{cooper@math.ucsb.edu}
\email[]{mark.baker@univ-rennes1.fr}

\thanks{Cooper was supported in part by NSF grant DMS--0706887 and CNRS}
\thanks{The authors thank IHP for hospitality during the completion of this paper.}

\begin{abstract} If $F$ is a surface with boundary, then a finitely generated subgroup without peripheral elements of  $G=\pi_1F$
  can be separated from finitely many other elements of $G$
by a finite index subgroup of $G$  corresponding to a finite cover $\tilde F$ with the same number of boundary components as $F$. 
\end{abstract}

\maketitle


Suppose $F$ is a compact, orientable surface with nonempty boundary.
A non-trivial element of $\pi_1(F)$ is {\em peripheral} if it is represented by a loop freely homotopic into $\partial F$.
A covering space $p:\tilde{F}\longrightarrow F$ is called {\em conservative} if $F$ and $\tilde F$
have the same number of boundary components: $|\partial F|=|\partial \tilde{F}|$.


\begin{theorem}[Main theorem]
\label{strongsepthm}
 Let $F$ be a compact, connected, orientable surface with $\partial F\ne \phi$ and $H\subset \pi_1(F)$ a finitely generated subgroup. Assume that no element of $H$ is peripheral. Given a (possibly empty) finite subset $B\subset\pi_1(F)\setminus H$, there exists a finite-sheeted cover $p:\tilde F  \longrightarrow F$ such that:
 \begin{itemize}
\item[i)] There is a compact, connected, $\pi_1$-injective subsurface $S\subset \tilde F$ such that $p_*(\pi_1S) = H$.
 \item[ii)] $p_*(\pi_1\tilde F)$ contains  no element of $B$.
\item[iii)] $\tilde{F}\setminus S$ is connected
and $incl_*: H_1(S)\rightarrow H_1(\tilde F)$ is injective.
\item[iv)] The covering is conservative.
\end{itemize}
\end{theorem}

This theorem, without (iii), is due to Masters and Zhang \cite{MZ1} and is a key ingredient in their proof
that cusped hyperbolic $3$-manifolds contain quasi-Fuchsian surface groups \cite{MZ1}, \cite{MZ2}.
Without  (iii), (iv)  the theorem is a special case of well-known theorems on subgroup 
separability of free groups \cite{H} 
and surface groups \cite{S}, \cite{S2}.  For a discussion of subgroup separability and 3-manifolds, see \cite{LongReid}.
   
The proof in \cite{MZ1}
uses the folded graph techniques due to Stallings, see  \cite{KM}. 
The shorter proof below uses cut and cross-join of surfaces. 
A cover is called {\em good} if properties (i)-(ii) hold and {\em very good} if (i)-(iii) hold.
The idea is to start with a good cover and then 
pass to a second cover which is very good. 
Then {\em cross-join operations} (defined below) are used to reduce the number of boundary
components of a very good cover until it is conservative. 

 \section{Constructing a Very Good Cover}
 
We first explain a geometric condition on a cover which ensures  it is good, and then 
use \ref{pi1injectivesurface} to construct a very good cover.

Choose a basepoint $x$ in the interior of $F$
and suppose $p:\tilde{F}\longrightarrow F$ is the cover corresponding to $H$. There is a compact, connected,
 incompressible subsurface $S$ in the interior of $\tilde{F}$  
 which is a retract of $\tilde{F}$ and which contains a lift $\tilde{x}$
 of $x$. 
 Each element  $g\in\pi_1(F,x)$ determines a unique  lift $\tilde{x}(g)\in\tilde{F}$ of the basepoint $x$. 
The surface $S$ can be chosen large enough to contain $\{\ \tilde{x}(b):b\in B\ \}$. 
Then $p|_S:S\longrightarrow F$ is a local homeomorphism. 
The work of M. Hall \cite{H} and P. Scott \cite{S} shows there is a  finite cover
$F'\longrightarrow F$ such that $S$ lifts to an embedding in $F'$.

If $\pi:F'\longrightarrow F$ is any cover
and there is a lift of $p|_S$ to $\theta:S\longrightarrow F'$ (thus $\pi\circ\theta=p|_S$) which is injective,
 we say {\em $S$ lifts to an embedding in the cover $F'$}.
\begin{proposition}[good cover]\label{first2} With the hypotheses of the main theorem,
 if $\pi:F'\longrightarrow F$ is any cover and $S$ lifts to an embedding in $F'$, then the cover is good.
\end{proposition}
\begin{proof} With the notation above, 
a based loop representing an element  $b\in B$ lifts to a path in $F'$
that starts at the basepoint $\tilde{x}\in Y=\theta(S)$ but ends at some other point $\tilde{x}(b)\ne \tilde{x}$ 
 in $Y$.\end{proof}

 \begin{addendum}[very good cover]\label{sepadd} There is a very good cover  $\tilde{F}$ of finite degree with
 $|\partial \tilde F|$ is even.
 \end{addendum}
 \begin{proof} We start with a good cover $F'$ of $F$ with finite degree and the subsurface 
 $S\subset F'$ described above and then construct a cover
 of $F'$ with the required property.   For notational elegance, we rename the first cover $F'$  as $F$. 
Let $p:\tilde{F}\longrightarrow F$ be the regular cover given by the kernel of the map of $\pi_1F$ onto  $H_1(F,S;{\mathbb Z}/2)$. 
  There is a lift  $\tilde{S}$ of $S$ to this cover by construction. The connectedness of $\tilde{F}\setminus \tilde{S}$ and the injectivity of $incl_* : H_1(\tilde S) \rightarrow H_1(\tilde F)$ are shown in theorem \ref{pi1injectivesurface} below.
Finally, since $\partial F\ne\phi$ and $S$ has no peripheral elements, it follows that $H_1(F,S;{\mathbb Z}/2)\ne 0$  so the cover
has even degree. Thus $\chi(\tilde{F})$ is even, therefore $|\partial \tilde{F}|$ is even.
   \end{proof}


The following allows us to lift a $\pi_1$-injective subsurface to a regular cover where it is $H_1$-injective and non-separating.
\begin{theorem}\label{pi1injectivesurface}
Suppose $F$ is a compact, connected, orientable surface, possibly with boundary,
which contains  a compact, connected, subsurface  $S$. 
Assume  that no component of $\cl(F\setminus S)$ is a disc or a boundary parallel annulus.
Let $p : \tilde F \rightarrow F$ be the cover corresponding to the kernel of
the natural homomorphism of $\pi_1F$ onto $G=H_1(F,S;{\mathbb Z}/2)$. 
If  $\tilde S_0$ is a connected component of $p^{-1}(S)$ then  $\tilde{F}\setminus\tilde{S}_0$ is connected.
Hence $incl_* : H_1(\tilde S_0 )\rightarrow H_1( \tilde F)$ is injective.
\end{theorem}
\begin{proof} We may assume $S\ne F$.
Define $X = $ \rm{cl} $(\tilde F \setminus \tilde S_0)$.
Let $Y$ be a connected component of $X$.  
We claim  that  $p(Y)\supset S$. Otherwise $p|_Y:Y\longrightarrow \cl(F\setminus S)$. Since $p|(Y\cap \tilde{S}_0)$ is injective it follows that $p|Y$ is injective,
thus $Y$ is 
 a lift of a component $Z$ of $\rm{cl}(F\setminus S)$.  
 
 If $Z\cap S$ is connected, then since $Z$ is not a disc or boundary parallel annulus, the image of 
 $H_1(Z)$ in $G$ is not trivial. 
 Thus $Z$ does not lift to the $G$-cover, a contradiction.
  
Hence $Z\cap S$ contains at least two distinct circle components $B_1,B_2$.
 There is a loop $\alpha=\beta\cdot\gamma\subset F$ which is the union of  two arcs connecting $B_1$ and $B_2$: 
 one arc  $\beta\subset Z$ and one arc  $\gamma\subset S$. Since $\alpha$
 has non-zero algebraic intersection number with the boundary component $B_1$ of $S$ 
 it is a non-zero element of $G$. 
 It follows  that the lift $\tilde{\beta}\subset\tilde{S}_0$ 
 of $\beta$ has endpoints in different components of $p^{-1}(S)$, since
 otherwise $\alpha$ would lift to a loop. But $\partial\tilde{\beta}\subset\partial Y\subset\partial\tilde{S}_0$
 which is a contradiction. Thus $p(Y)\supset S$.
 
 Choose some Riemannian metric on $F$. This metric pulls back to one on  $\tilde{F}$  
 which is preserved by covering transformations. 
If $X$ is not connected, let $Y$ be a component of smallest area.  
It follows that $Y$ contains some component $\tilde{S}_1\ne\tilde{S}_0$ of $p^{-1}(S)$ in its interior. 
However the cover is regular
so there is a covering transformation $\tau$ taking $\tilde{S}_0$ to $\tilde{S}_1$. 
Thus $\tau$ takes components of $\tilde{F}\setminus \tilde{S}_0$
to components of $\tilde{F}\setminus \tilde{S}_1$. One of these components 
contains $\tilde{S}_0$ so the other ones are strictly 
contained in $Y$ which contradicts that $Y$ has minimal area. Hence $X=Y$ is connected.

For the last conclusion, apply Mayer-Vietoris to $\tilde F = \tilde S_0 \cup X$ with 
$ \tilde S_0 \cap X = \partial \tilde S_0 \cap \partial X$.
Since $X$ is connected, if  the kernel of  $i_*: H_1(\tilde S_0) \rightarrow H_1(\tilde F)$
 is nontrivial, then $\partial X\varsubsetneq \partial \tilde S_0$. This implies
$\partial X\cap\partial \tilde F=\phi$. Since $p(X)\supset S$ it follows that 
$\partial \tilde F=\phi$. But then $\partial\tilde S_0=\partial X$, hence the kernel is trivial.
\end{proof}

\section{Cross-Joining  Covers}

Suppose $F$ is a surface and $\alpha_1$ and
 $\alpha_2$ are disjoint arcs
 properly embedded in $F$. Let $N(\alpha_i)\equiv \alpha_i\times[-1,1]$ be disjoint regular neighborhoods
  of the arcs $\alpha_i$ in $F$ 
  such that $\alpha_i\equiv\alpha_i\times 0$ and $N(\alpha_i)\cap \partial F=(\partial\alpha_i)\times[-1,1]$.
  The sets $\alpha_i\times(0,\pm 1]\subset F$ are called the {\em $\pm$  sides} of $\alpha_i$.

 Given a homeomorphism $h:N(\alpha_1)\longrightarrow N(\alpha_2)$  taking the $+$ side
 of $\alpha_1$ to the $+$ side  of $\alpha_2$,  the {\em cross-join} of $F$ along $(\alpha_1,\alpha_2)$ is the surface
 $K$ defined as follows. The surface $F^-=F\setminus(\alpha_1\cup\alpha_2)$ 
 contains four subsurfaces
 $\alpha_i\times(0,\pm1]$.
   Let  $F^{cut}$ be the surface obtained
 by completing these subsurfaces to $\alpha_i\times[0,\pm1]$. Thus $F^{cut}$
 has two copies $\alpha_i^+,\alpha_i^-$ of $\alpha_i$ in $\partial F^{cut}$ and identifying these copies suitably
  produces $F$.
The surface $K$ is the quotient of $F^{cut}$ obtained by using
 $h$ to identify $\alpha_1^-$ to $\alpha_2^+$ and $\alpha_1^+$ to $\alpha_2^-$. Note that here
 we do not require $F$ to be connected, so that $\alpha$ and $\beta$ might be in different
 components of $F$. 
 
 There are two special cases of   cross-join which 
 will be used to change the number of boundary components of a surface:

 
 
 \begin{lemma}\label{cross-joinboundarycount} Suppose the compact surface
  $F$ contains two disjoint properly embedded arcs 
 $\alpha$ and $\beta$. In addition suppose:
 \begin{itemize}
\item[(1)] either $F$ is connected and the endpoints of $\alpha,\beta$ lie on four distinct components of $\partial F$
\item[(2)] or  $F$ is the union of two connected components $A$ and $B$ and $\alpha\subset A$ 
  has both endpoints on the same boundary component
 and $\beta\subset B$ has endpoints on distinct boundary components. 
 \end{itemize}
   Then a surface $K$ obtained by cross-joining
along these arcs  has $|\partial K|=|\partial F|-2$. Furthermore 
 $\chi(K)=\chi(F)$ and  $K$ is connected.
  \end{lemma}  
  \begin{proof} We verify that $K$ is connected. In the first case this
  follows since the arcs do not disconnect the boundary components on which they have
  endpoints; therefore $F\setminus(\alpha\cup\beta)$ is connected. 
  In the second case it follows because $B\setminus\beta$ is connected, and every
  point in $K$ is connected to a point in this subset by an arc.\end{proof}
  

Suppose $p:\tilde{F}\longrightarrow F$ is a (possibly not connected) 
covering of surfaces and $\alpha$ is an arc properly embedded
in $F$. Suppose $\tilde{\alpha}_1$ and $\tilde{\alpha}_2$ are two distinct lifts of $\alpha$ to $\tilde{F}$; then
they are disjoint.
The map $p$ provides a homeomorphism between small regular neighborhoods of these
two arcs. Using this to cross-join  produces a surface $\tilde F'$ and since the identifications are
compatible with $p$ there is 
a covering map $p':\tilde{F}'\longrightarrow F$.  

An important special case is when $\tilde{F}$ is a $(d+1)$-fold cover
which is the disjoint union of  a $1$-fold cover $F_1\longrightarrow F$  
and some connected $d$-fold cover $F_d\longrightarrow F$. Then
cross-joining an arc in $F_1$ with one in $F_d$ produces a 
cover of degree $d+1$. 

 To produce a new cover $F'$ of $F$ by  a cross-join along two arcs   in some cover $\tilde F$
 requires the arcs are disjoint from each other. If $S$ is embedded in $\tilde{F}$ and 
 these arcs are also disjoint from $S$, then $S$ lifts to an embedding in $F'$, so
the cover $F'$ is  good.
We call the combination of these two properties the {\em disjointness condition}.

There is a metric condition, 
involving some arbitrary choice of Riemannian metric on $F$, that ensures the disjointness condition
is satisifed
and therefore that the new cover is good. The  next
lemma provides a uniform upper bound on the lengths of the arcs we will use to cross-join
 {\em in any cover} of $F$. 

  \begin{lemma}[short arcs]\label{arclemma} Suppose $F$ is a compact, connected surface with 
  a Riemannian metric such that the maximum distance between points in $F$ is $\ell$. If $\tilde{F}$ is a finite cover of 
  $F$ then\begin{itemize}
\item[1] If $A$ and $B$ are distinct components of $\partial F$ then there is an arc $\alpha$ in $F$ connecting them
   and $length(\alpha)\le \ell$.
\item[2] If some component $A$ of $\partial F$ has (at least) two pre-images in $\partial\tilde{F}$ 
then there is an embedded arc $\alpha$
  in $F$ of length at most $2\ell$ which lifts to an arc with endpoints on distinct pre-images of $A$.
  \end{itemize}
\end{lemma}
\begin{proof}  The first claim is obvious. For the second claim, since every point in $\tilde{F}$ is within a distance at most $\ell$ 
of some point in $p^{-1}(A)$ and $\tilde F$ is connected, some point in $\tilde F$ is within a distance at most $\ell$ of points in
two distinct components of $p^{-1}(A)$. This gives an arc $\beta$ in $\tilde{F}$ of length at most $2\ell$ which
connects two distinct components of $p^{-1}(A)$. 

Let $\gamma:[0,2R]\longrightarrow \tilde{F}$ be a shortest arc 
connecting two distinct components of $p^{-1}(A)$ and parameterized by arc length. Then $R\le \ell$.
To complete the proof we show that $\gamma$ projects to an embedded arc in $F$.
Observe that $$d_{\tilde{F}}(\gamma(t),p^{-1}(A))=\min(t, 2R-t)$$ otherwise there is a shorter arc connecting two distinct components of $p^{-1}(A)$.
It follows that $$d_F(p(\gamma(t)),A)=\min(t, 2R-t)$$
This means that the distance in $F$ of a point on $p\circ\gamma$ from $A$ is given by arc length along $p\circ \gamma$. It follows
  that $\alpha=p\circ \gamma$ is the required embedded arc.   \end{proof}

An arc of length at most $2\ell$ is called {\em short}.  The next lemma provides a 
conservative cyclic cover with large diamater
of a  surface $F$.
If a short arc in $F$ connects two distinct boundary components, then so does
every covering translate of it. If $S$ lifts to the cover then there are many different translates
of the short  arc  that are far from each other 
and far from the lift of $S$. 
In particular the disjointness condition is satisfied by suitable translates of a lifted short
arc in this cover.

\begin{lemma}[big covers] \label{bigcover} Suppose $F$ is a compact connected surface
with $k\ge 2$ boundary components and which contains a compact, connected, incompressible subsurface $S\subset\interior(F)$ with $F\setminus S$ connected.
 Given $n>0$ there is a conservative finite cyclic cover $\tilde{F}\longrightarrow F$
of degree bigger than $n$ and a lift,  $\tilde{S}$, of $S$ to $\tilde F$.
Furthermore $\tilde{F}\setminus\tilde{S}$ is connected.
\end{lemma}
\begin{proof} Let $Y$ be the surface obtained from $F\setminus\interior(S)$ by gluing
a disc onto each component of $\partial S$.   Then $Y$ is a connected surface
with  $k$ boundary components and there is a natural isomorphism of $H_1(F)/H_1(S)$ onto $H_1(Y)$. Choose a prime $p>\max(k,n)$.
Because $Y$ is connected, there is an epimorphism from $H_1(Y)$ onto ${\mathbb Z}/p$ which sends one component of
$\partial Y$ to $k-1$ and all the other $(k-1)$ components of
 $\partial Y$ to $-1$. 
Now $(k-1)$ is coprime to $p$ because  $2\le k<p$. Therefore this defines a {\em conservative} cyclic $p$-fold cover $\tilde Y$
of $Y$. It also determines a conservative cyclic $p$-fold cover
 of $F$ such that $S$ lifts. Since ${\tilde Y}$ is connected it follows that $\tilde{F}\setminus \tilde{S}$
 is connected.
    \end{proof}
  

\section{Proof of main theorem}
In this section all covers are of finite degree.
 Given a  cover $p:\tilde{F}\longrightarrow F$ the {\em excess} number of boundary components $E(p)$ for this
cover is defined as $E(p)=|\partial \tilde F| - |\partial F|$.
By \ref{sepadd} there is a very good cover $p: \tilde F \rightarrow F$ with $|\partial \tilde F|$ is even.
If  $E(p)=0$ the theorem is proved.  

We first use \ref{bigcover} 
to replace a very good cover cover by
another very good cover with the same excess  
and where there are lifts of a short arc that are far apart. Then we change the
cover with a cross-join that reduces the excess. To apply \ref{bigcover} 
requires that $F\setminus S$ is {\em connected}. 
We must verify this property continues to hold after the cross-join so the process can be repeated.
First observe that the cyclic cover produced by \ref{bigcover} leaves $F\setminus S$ connected.

In each case  (except the last one) we will
use one of the two cross-joins described in \ref{cross-joinboundarycount} to produce
a new {\em connected} cover $F'$ of $F$ and a lift $\tilde{S}$ of $S$.
Since the cross-join arcs are disjoint from $S$ they also determine a {\em connected} cover of $F\setminus S$.
Thus $F'\setminus\tilde{S}$ is connected, as required. 
This implies $incl_*: H_1(\tilde{S})\rightarrow H_1(F')$ is injective, so the new cover is also very good.

\medskip
{\em  \noindent  Case  when $|\partial F|=1$.}\\
By \ref{arclemma} there is a properly embedded,  short arc, $\alpha$, in $F$ 
which is covered by an arc $\beta$ with endpoints on two distinct
boundary circles of $\partial \tilde F$. 
There is a conservative cyclic cover  of $\tilde{F}$, which we also denote by $\tilde F$, to which $S$ lifts
 with diameter much larger than the length of $\beta$
and the diameter of $S$.
Thus there is a lift of $\beta$ which is disjoint from $S$. Since the cover is conservative
every lift of $\beta$ connects (the same pair of) distinct boundary components.

Cross-join $(\tilde F, \beta)$ with $(F,\alpha)$ to obtain a 
cover $F'$ with one fewer boundary circle than $\tilde F$. There is a lift
of $\tilde{S}$ to $F'$ and $F'\setminus \tilde{S}$ is connected.
 Repeat the process until the cover has only one boundary component. 
 This completes the proof when $|\partial F|=1$.
\medskip

{\em \noindent Case when $|\partial F| \geq 2$.}\\
First we show how to make $E(p)$ even by performing a  cross-join if needed.
This first step will increase the number of boundary components.

 Suppose $E(p)$ is odd. By \ref{sepadd}  $|\partial\tilde{F}|$ is even, so $|\partial F|$ is odd.
We can make $E(p$) even by
 cross-joining $(\tilde F,\tilde{\alpha})$ and $(F,\alpha)$ to obtain a cover $p':F'\longrightarrow F$.
    To perform the cross-join choose a short embedded arc $\alpha \subset F$
with endpoints on two distinct circles $C$, $C'$ of $\partial F$. Choose a lift of 
 $\tilde{\alpha} \subset \tilde F$,  with endpoints on two preimages $\tilde C$, $\tilde C'$. 
By the big cover lemma \ref{bigcover} we can choose $\tilde{\alpha}$ disjoint from $S$ in $\tilde F$. 
Then $F'$ is the  cross-join of $(F,\alpha)$ and $({\tilde F},{\tilde\alpha})$. The surface $S$ lifts
to $F'$ and $F'\setminus S$ is connected by \ref{cross-joinboundarycount}.

Here is the outline of the rest of the proof. If $E(p)\ne0$ then it is even. 
We proceed as follows using suitable cross-joins to construct
new coverings. If there are two different components $C,C'\subset\partial F$ which both have more
than one pre-image in $\partial\tilde{F}$ then we find a short arc $\alpha$ in $F$ connecting $C$ and $C'$ and cross-join
$\tilde{F}$ to itself along two suitable lifts of $\alpha$ in $\tilde{F}$. This reduces the excess by $2$. 
After finitely many steps we obtain a cover so that at most one component $C\subset\partial F$ has more
than one pre-image. A single {\em cyclic cross-join} (defined below) is done simultaneously to reduce the excess to zero.
Here are the details.

Suppose $A$ and $B$ are distinct circles in $\partial F$ which both have (at least) two distinct pre-images 
$\tilde{A}_i,\tilde{B}_i$ for $i=1,2$ in $\partial \tilde F$. 
Choose a short arc $\gamma$ in $F$ with endpoints on $A$ and $B$. Let $\alpha_i$ be
a lift of $\gamma$ with  one endpoint on $\tilde{A}_i$ and $\beta_i$ a lift
with an endpoint on $\tilde{B}_i$. Inductively we assume that $\tilde F\setminus S$
is connected. Replace $\tilde{F}$ by   a large cyclic conservative cover such that these
arcs are all far apart and far from $S$. 
Thus there is a cover obtained by cross-joining  along any
pair of distinct arcs chosen from this set of four and  $S$ lifts to this cover.

We claim that there is a pair of these arcs
which have endpoints on four distinct boundary components of $\tilde F$.
It follows from lemma \ref{cross-joinboundarycount} that cross-joining along
this pair reduces the excess by $2$ and $S$ lifts to the cover $F'$ so produced. 
Furthermore, since $\tilde{F}\setminus S$ is connected
 it follows that $F'\setminus S$ is connected by \ref{cross-joinboundarycount}.

If $\alpha_1$ and $\alpha_2$ do not both have endpoints on the same lift $\tilde{B}$ of $B$ the pair $(\alpha_1,\alpha_2)$ works.
Similarly if $\beta_1$ and $\beta_2$ do not both have endpoints on the same lift $\tilde{A}$ of $A$
the pair $(\beta_1,\beta_2)$ works. 
The remaining case is (after relabelling) $\alpha_1$ and $\alpha_2$ both have endpoints on 
a component $\tilde{B}\ne \tilde B_2$
which covers $B$ and
$\beta_1,\beta_2$ both have endpoints on some component $\tilde{A}\ne \tilde A_2$ which covers $A$.
 Then $\alpha_2$ connects $A_2$ to $\tilde{B}\ne \tilde B_2$ and $\beta_2$
connects $B_2$ to $\tilde{A}\ne \tilde A_2$. Thus the pair $(\alpha_2,\beta_2)$ works.

 Repeating this process a finite number of times reduces the excess by an even number until   either  
$ |\partial \tilde F| = |\partial F |$  or else there is a unique component $C$ of $\partial F$ with
 more than one pre-image. 
 In the latter case the excess is even so there is an an odd number of pre-images $p^{-1}(C)= \{C_0,...,C_{2k}\}$.
 
Refer to figure \ref{ringpic}.  Choose  a component 
$A$ of $\partial\tilde F$ that does not cover $C$. This is possible because $|\partial F|\ge 2$.
 Let $\beta$ be a short arc in $F$ with endpoints
on $p(A)$ and $C$.
For each $i$ there is a lift $\beta_i$ of $\beta$ with one endpoint on $C_i$
and the other on $A$. As before we may assume all these lifts are far apart and far from $S$. 
Orient each arc $\beta_i$  so it points from $A$ to $C_i$  and call the left side $+$  and the right side $-$. 
 Now cross-join cyclically as follows. Cut $\tilde F$ along the union of these arcs and 
 join the $-$ side of $\beta_i$ to the $+$ side of $\beta_{i+1}$, with all integer subscripts taken mod $2k+1$. 

\begin{figure}[ht]
 \begin{center}
 \psfrag{C}{$C_0$}
\psfrag{D}{$C_1$}
\psfrag{E}{$C_2$}
\psfrag{F}{$C_3$}
\psfrag{G}{$C_{4}$}
\psfrag{m}{$-$}
\psfrag{p}{$+$}
\psfrag{R}{$E_0$}
\psfrag{S}{$E_1$}
\psfrag{T}{$E_2$}
\psfrag{U}{$E_3$}
\psfrag{V}{$E_{4}$}
\psfrag{A}{$A$}
\psfrag{qa}{$u_0^+$}
\psfrag{qb}{$u_1^-$}
\psfrag{qc}{$u_1^+$}
\psfrag{qd}{$u_2^-$}
\psfrag{qe}{$u_2^+$}
\psfrag{qf}{$u_{3}^-$}
\psfrag{qg}{$u_{3}^+$}
\psfrag{qh}{$u_{4}^-$}
\psfrag{qi}{$u_{4}^+$}
\psfrag{qj}{$u_0^-$}
\psfrag{ba}{$\beta_0$}
\psfrag{bb}{$\beta_1$}
\psfrag{bc}{$\beta_2$}
\psfrag{bd}{$\beta_3$}
\psfrag{be}{$\beta_{4}$}
\psfrag{NN}{The preimage of $A$}
	 \includegraphics[scale=0.8]{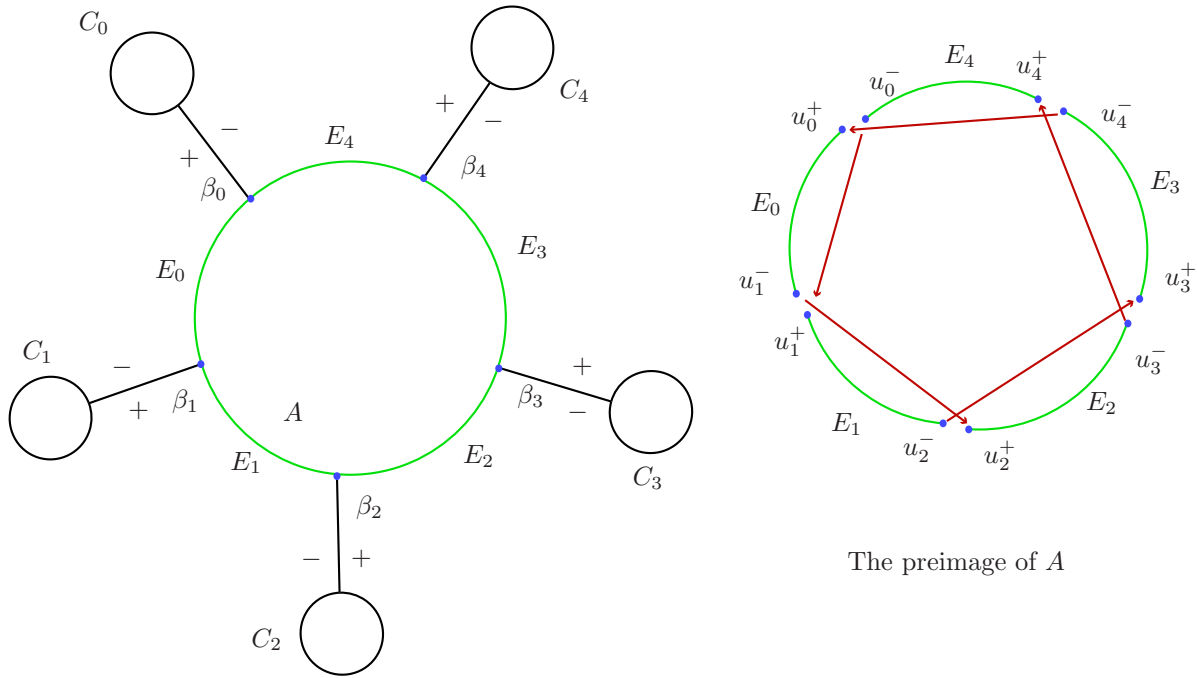}
 \end{center}
\caption{Cyclic cross-joining, $2k+1=5$ illustrated}	 \label{ringpic}
\end{figure}

  The resulting cover has a single pre-image of $C$. Indeed, each $C_i$ has been cut at one point to give an interval
$D_i=[t_i^+,t_i^-]$ where the label $i$ denotes an endpoint of $\beta_i$ and $t_i^{\pm}$ is 
on the $\pm$ side of $\beta_i$. 
These intervals are then
 glued by identifying $t_{i}^-$ in $D_i$ to $t_{i+1}^+$ in $D_{i+1}$.
 The result is obviously connected: a single circle.

To analyse the preimage of $p(A)$ the circle $A$ was cut at $2k+1$ points
 to produce $2k+1$ subarcs $E_i=[u_i^+,u_{i+1}^-]$ where $u_i^{\pm}$ is on the $\pm$ side of $\beta_i$. 
 Then  $E_{i}$ is glued  to $E_{i+2}$ by identifying $u_{i+1}^-$ with $u_{i+2}^+$ (see figure 1).
Since there are $2k+1$ intervals and the $i$'th one is glued to the $(i+2)$'th one the result is connected because $2$ is coprime to $2k+1.$ 
 This gives the required conservative cover completing the proof of the main theorem.

 \small
\bibliography{refs} 
\bibliographystyle{abbrv} 

\end{document}